\documentclass[12pt,reqno,a4paper,openany]{amsbook}
\usepackage{extsizes}
\usepackage{blindtext}

\usepackage{tikz}
\usepackage{hyperref}
\usepackage{amsmath}
\usepackage{graphicx}

\usepackage{amsfonts}
\usepackage{amssymb}
\usepackage{setspace,kantlipsum}
\usepackage{dirtytalk}
\usepackage[english]{babel}

\tikzset{node distance=2cm, auto}

\newtheorem{thm}{Theorem}[section]
\newtheorem{prop}[thm]{Proposition}
\newtheorem{lma}[thm]{Lemma}
\newtheorem{cor}[thm]{Corollary}

\theoremstyle{definition}
\newtheorem{defn}[thm]{Definition}
\newtheorem{eg}[thm]{Example}

\theoremstyle{remark}

\newtheorem{rmk}[thm]{Remark}

\title{A commutative model for PL compactly supported cohomology in charactersitic zero}
\author{Tom Sutton}
\email{tombsutton@gmail.com}

\begin{document}

\begin{abstract}
In this paper, we show that that classical rational homotopy theory in the sense of Sullivan \cite{sullivaninfinitesimal} can be extended compactly supported setting.  This presents a simplicial version of the compactly supported de Rham complex in characteristic zero, and proving that it models singular compactly supported cohomology.  This presents new avenues of possible study in proper homotopy theory, and further extensions of the ideas within the classical literature of Quillen \cite{quillenrht} and Sullivan.
\end{abstract}

\maketitle


\section*{Introduction}

Throughout this paper, given a simplicial set $X$, and a collection of simplices $S$ of $X$, we denote by $<S>$ the simplicial subset of $X$ generated by $S$.  If $X$ is a simplicial set, and we write $K\subset X$, then unless otherwise stated, we will assume that $K$ is a subsimplicial set, rather than just a collection of simplices in $X$ (although on at least one occasion it will simply be the latter).

\vspace{3mm}

In this paper, we show that Sullivan's method of PL polynomial forms for rational homotopy theory naturally extends to the compactly supported setting: in particular giving an explicit commutative model for compactly supported cohomology, that applies to all CW complexes, rather than just manifolds.  An unfortunate feature of the category of manifolds is that many colimits do not exist in it, and this has led various authors to consider generalisations, the earliest of which was Chen's \emph{Chen spaces} in 1973.  This was followed by Souriau's \emph{diffeological spaces} in 1980.  The categories of both are complete and cocomplete, and have other nice categorical properties (locally cartesian closed, weak subobject classifier).  This interest beyond the category of manifolds means it is both interesting and helpful to see how much of the theory applying to manifolds can be carried over in some way.  In particular, in \cite{haraguchi}, Haraguchi develops a theory of compactly supported cohomology for diffeological spaces.  Our work in this chapter can be viewed as a simplicial analogue of Haraguchi's work.

\vspace{3mm}

We also feel that this work has potentially deep proper homotopical implications.  Sullivan's original polynomial de Rham complex gives (under an array of conditions) an equivalence of two homotopy theories, and so one might expect that a compactly supported de Rham complex could give an equivalence between certain \say{proper homotopy theories}.  The sense in which \say{proper homotopy theory} would be meant is not completely clear, although the work of Baues and Quintero in \cite{infhom} seems appropriate, and would need to be reworked for simplicial sets.  The fundamental change in moving from homotopy theory to proper homotopy theory is that the notion of a model structure is no longer applicable, as the category of spaces with proper maps is far from having a natural model structure (it doesn't even have a terminal object).  Instead, the category of spaces and proper maps has the structure of a \emph{cofibration category}, and this is the axiomatic framework in which \cite{infhom} largely works.  We have not had the time to pursue these homotopical questions, but think this would be a very interesting future project.

\vspace{3mm}

We will now begin the chapter by setting up the basic notions, and reminding the reader of important earlier constructions.

\vspace{3mm}
Recall the simplicial CDGA over $k$ denoted by $\nabla(*,*)$, defined in \hyperref[nabla]{\ref*{nabla}}.
%
%
%
%

\begin{defn}
A simplicial set $X$ is $\emph{finite}$ if it has only finitely many non-degenerate simplices, and is $\emph{locally finite}$ if every simplex is a face of only finitely many non-degenerate simplices.

\end{defn}

\begin{defn}
Let $X$ be a simplicial set.  $\nabla(*,*)$ allows the authors of \cite{bousgug} to define the CDGA $A^*X$of polynomial forms on $X$, by
\[A^qX:=sSet(X,\nabla(*,q))\]
Correspondingly, we define the CDGA $A_c^*X$ of $\emph{compactly supported}$ polynomial forms on $X$ by
\[A_c^qX:=\{ \Phi\in sSet(X,\nabla(*,q))|\exists\text{ finite } K\subset X\text{ s.t }\Phi|_{<X\setminus K>}=0\} \]
Given $\phi \in A^qX$ and a $q$-simplex $\sigma$ of $X$, we will often write $\phi |_{\sigma}$ in place of $\phi (\sigma )$.

\end{defn}

\begin{rmk}
$A^qX$ can be thought of as all possible ways of assigning a $q$-form to each simplex of $X$, in a manner which is compatible with the face and degeneracy operators.  Thus $A^qX$ is analagous to the global sections of a sheaf of functions on a space.  It is easily checked that $A^*X$ and $A_c^*X$ are both well defined CDGAs.  However, if $X$ is not finite, the latter does not have any unit element as non-zero constant $0$-forms are not compactly supported.  This can be re-interpreted as the fact that the category of spaces with proper continuous maps has no terminal object.

\end{rmk}

We now give what will be an essential property of $\nabla (*,q)$ for our purposes.  We will refer to this property as the \emph{extension property}.  The geometric interpretation of this property is that if one has a $q$-form defined on the boundary of simplex, then that form can be extended to a $q$-form on the entire simplex.  Expressing this for $\nabla(*,q)$, the property states that if we are given forms $\omega_0,...,\omega_p\in \nabla (p-1,q)$ that model a form on the boundary of a $p$-simplex (the necessary condition for this is that $\partial_i\omega_j=\partial_{j-1}\omega_i$, for all $i$ and $j$), then there exists a form $\omega \in \nabla (p,q)$ with $\partial_i\omega =\omega_i$, for all $i$.  In addition, this can be done in a way which is linear with respect to addition of forms.

\vspace{3mm}

We now state the extension property described above in its precise form.

\begin{prop} (The Extension Property.  Corollary $1.2$ of \cite{bousgug})\\
\label{extension_prop}
There exists a naturally defined function
\[E:\{ (w_0,...,w_p)|w_k\in \nabla(p-1,q)\text{ for all }k\text{, }\text{ }\partial_iw_j=\partial_{j-1}w_i\text{ for all }i\le j\}\to \nabla(p,q)\]
such that
\[\partial_i(E(w_0,...,w_p))=w_i\]
for all $i$, and
\[E(w_0,...,w_p)+E(w_0',...,w_p')=E(w_0+w_0',...,w_p+w_p')\]

\end{prop}
\hyperref[extension_prop]{\ref*{extension_prop}} is the $\emph{extension property}$ for $\nabla(*,*)$, and geometrically means that if we have a form defined over the boundary of a simplex, we can extend it to the entire simplex.

\vspace{3mm}

Bousfield and Gugenheim in \cite{bousgug} go on to prove that $\nabla(p,*)$ is acyclic (the Poincar\'{e} lemma), and they also define a process of formal integration of forms over a simplex, which we review now.

\begin{defn}
\label{int_defn}
Suppose $w\in \nabla(p,p)$ is given by $w=f(t_1,...,t_p)dt_1...dt_p$, where $f$ is a polynomial.  Let $|\Delta^p|$ denote the standard $p$-simplex in $\mathbb{R}^{p}$ given by $0\le t_i\le 1$ and $0\le t_1+...+t_p\le 1$.  Then since $k$ has characteristic $0$ we can compute $\int_{|\Delta^p|}fdt_1...dt_p$ term by term as an integral over $\mathbb{R}$, the answer being a polynomial with coefficients in $k$, and so we define
\[ \int w:=\int_{|\Delta^p|}fdt_1...dt_p\]

\end{defn}
There is a total differential
\[ \partial :\nabla(p,q)\to \nabla(p-1,q)\]
defined by $\partial=\Sigma_{i=0}^p(-1)^i\partial_i$, satisfying $\partial d=d\partial$.  $\partial$ should be thought of as simply restricting a form to its boundary.  We now have the following

\begin{prop} (Stokes' Theorem (Proposition $1.4$ of \cite{bousgug})\\
For any $w\in \nabla (p,p-1)$
\[ \int dw=\int \partial w\]

\end{prop}

	\section{PL bump functions}

One way of obtaining a Mayer-Vietoris sequence for $A_c^*$ is to first show that we have a PL analogue of bump functions.  The situation is made somewhat simpler than for smooth manifolds, since our functions can be piecewise smooth, however we are of course restricted within this to using piecewise polynomial functions.

\begin{defn}

Let $X$ be a simplicial set, and $K\subset X$ a simplicial subset.  Define the $\emph{minimal neighbourhood}$ of $K$ in $X$ to be
\[ \epsilon (K):=<\{ \sigma \in X|\partial_{i_1}...\partial_{i_{m_\sigma}}\sigma \in K\text{ for some $\{ i_j\} $}\} >\]
\end{defn}

\begin{eg}
Consider when $X$ is the standard tessellation of the plane using equilateral $2$-simplices.  This is in fact a simplicial complex, but we can view it as a simplicial set where the faces of each non-degenerate simplex are also non-degenerate.  In this case, if we take $L\subset X$ to a be a single vertex $v$, then $\epsilon (L)$ is a hexagon with $6$ non-degenerate $2$-simplices, all meeting $v$.

\end{eg}

\begin{thm} (Existence of PL bump functions)
\label{bump_functions}
Let $X$ be a simplicial set, and $L\subset K\subset X$ be subsimplicial sets such that $\epsilon (L)\subset K$.  Then there exists some $\phi \in A^0(X)$ such that $\phi |_L=1$ and $\phi |_{<X\setminus K>}=0$.

\end{thm}
\begin{proof}
It suffices to prove the theorem in the case that $K=\epsilon (L)$, because the $\phi $ constructed in the proof for the case $K=\epsilon (L)$ will satisfy $\phi |_L=1$ and $\phi |_{<X\setminus K>}=0$, for any $K$ as in the statement of the theorem.  So we need to construct $\phi $ so that that $\phi_{<X\setminus \epsilon (L)>}=0$.  We begin by defining $\phi $ on $<\epsilon (L)\setminus L>$.
For each $m\ge 0$, denote the set of non-degenerate $m$-simplices of $<\epsilon (L)\setminus L>$ by $\Sigma^m:=\{ \sigma_{\alpha}|\alpha\in I_m\} $, for some indexing set $I_m$.  For each $\gamma\in I_0$, define
  \begin{equation}
   \phi|_{\sigma_{\gamma}^0}=
    \begin{cases}
      1, \text{if } \sigma_{\gamma}^0\in L \\
      0, \text{otherwise}
    \end{cases}
\nonumber
  \end{equation}
Now suppose $\phi $ has been defined on all simplices of $\Sigma^k$, for each $k<n$, in such a way that $\partial_i(\phi |_{\sigma_{\gamma}^k})=\phi |_{\partial_i\sigma_{\gamma}^k}$, for all $i$, for all $\gamma\in I_k$ and $k\le n$, and also that
  \begin{equation}
   \phi |_{\sigma_{\gamma}^k}=
    \begin{cases}
      1, \text{if } \sigma_{\gamma}^k\in L \\
      0, \text{if for all }s\ge 0\text{, and for all } i_1,...,i_s\text{, we have } \partial_{i_1}...\partial_{i_s}\sigma_{\gamma}^k\notin L
    \end{cases}
\nonumber
  \end{equation}
Then $\phi $ can be naturally extended to be defined on all degenerate simplices of $<\epsilon (L)\setminus L>$ of dimension $\le n$.  Then for all $\gamma\in I_n$ we can define
  \begin{equation}
    \phi |_{\sigma_{\gamma}^n}=
    \begin{cases}
      1, \text{if } \sigma_{\gamma}^n\in L \\
      0, \text{if for all } s\ge 0, \text{and for all } i_1,...,i_s\text{, we have } \partial_{i_1}...\partial_{i_s}\sigma_{\gamma}^n\notin L\\
      \text{any extension, otherwise}
    \end{cases}
\nonumber
    \end{equation}
Where in the last case, such an extension exists by the \hyperref[extension_prop]{extension property \ref*{extension_prop}}.  Now for all $\sigma \in <X\setminus \epsilon(L)>$, define $\phi |_{\sigma }=0$.  To show this gives a well-defined extension of $\phi $, we need to check that it agrees with the face and degeneracy operators, and that it agrees with the above definition of $\phi$ on $\epsilon (L)$.  Suppose $\sigma \in <X\setminus \epsilon(L)>\bigcap \epsilon(L)$.  Assume first that $\sigma $ is non-degenerate: as $\sigma \in <X\setminus \epsilon(L)>$, there exists some $\tau \in X\setminus \epsilon (L)$ such that $\gamma_1\gamma_2...\gamma_N\tau =\sigma $, where each $\gamma_j$ is some face or degeneracy map (and potentially $N=0$).  Now we can use the simplicial identities reorder the $\gamma_j$ so that $s_{n_1}...s_{n_t}\partial_{i_1}...\partial_{i_s}\tau =\sigma$ (where $t+s=N$), and now since $\sigma $ was assumed to be non-degenerate, we must have $\partial_{i_1}...\partial_{i_s}\tau =\sigma$.  Now since $\tau \notin \epsilon(L)$, the equation relating $\sigma $ and $\tau $ means that for all $q\ge 0$ and all $p_1,...,p_q$, we have $\partial_{p_1}...\partial_{p_q}\sigma \notin L$ and hence $\phi |_{\sigma }=0$ is well-defined.  If $\sigma $ were instead degenerate, then since $<X\setminus \epsilon (L)>\bigcap \epsilon (L)$ is a subsimplicial set, there would exist some non-degenerate $\sigma '\in <X\setminus \epsilon (L)>\bigcap \epsilon (L)\subset \epsilon (L)$ such that $s_{j_1}...s_{j_l}\sigma '=\sigma $, and so since $\phi $ agrees with the face and degneracy operators on $\epsilon (L)$, setting $\phi |_{\sigma }=0$ is well-defined.  Now suppose only that $\sigma \in <X\setminus \epsilon (L)>$.  Then if $\partial_i\sigma \notin \epsilon(L)$ (respectively $s_j\sigma \notin \epsilon(L)$) then $\partial_i(\phi |_{\sigma })=\phi |_{\partial_i\sigma}=0$ (resp. $s_j(\phi |_{\sigma })=\phi |_{s_j\sigma}=0$).  So supposing $\partial_i\sigma \in \epsilon(L)$ (resp. $s_j\sigma \in \epsilon(L)$), we have that $\partial_i\sigma \in <X\setminus \epsilon(L)>\bigcap \epsilon(L)$ (resp. $s_j\sigma \in <X\setminus \epsilon(L)>\bigcap \epsilon(L)$), and so by the above argument $\partial_i(\phi |_{\sigma })=\phi |_{\partial_i \sigma }=0$ (resp. $s_j(\phi |_{\sigma })=\phi |_{s_j\sigma }=0$).\\
Hence $\phi \in A^*X$ satisfies the required conditions.

\end{proof}
	\section{Two contravariant Mayer-Vietoris Sequences}

Any reasonable compactly supported cohomology theory should have Mayer-Vietoris sequences associated to it.  We give two ways of obtaining such sequences, one using \hyperref[bump_functions]{\ref*{bump_functions}} and imposing a condition on the intersection (version 1), and the other using a decomposition of our given simplicial set as a pushout, with various conditions on the maps (version 2).  In both cases we obtain a contravariant Mayer-Vietoris sequence.

\begin{defn}
Let $X$ be a simplicial set with two subsimplicial sets $U, V\subset X$ which cover $X$.  Then $U$ and $V$ are said to have $\emph{good intersection}$ if $\epsilon (<V\setminus U>)\subset V$.

\end{defn}

\begin{lma}
Let $X$ be a simplicial set with subsimplicial sets $U, V\subset X$ which cover $X$.  Then
\[ \epsilon (<V\setminus U>)\subset V\Leftrightarrow \epsilon (<U\setminus V>)\subset U\]
and hence the notion of a good intersection is symmetric.

\end{lma}

\begin{proof}

Suppose $\epsilon (<U\setminus V>)\not\subset U$.  Then by the minimality of $<->$,
\[ \{ \sigma \in X|\partial_{i_1}...\partial_{i_{m_\sigma}}\sigma \in <U\setminus V>\text{ for some $\{ i_j\}$ }\}\not\subset U\]
and hence there exists $\sigma \in X\setminus U$ such that $\partial_{i_1}...\partial_{i_s}\sigma \in <U\setminus V>$.  So since $U$ and $V$ cover $X$, $\sigma \in V$ and
\[ \partial_{i_1}...\partial_{i_s}\sigma =(Xg)\gamma \]
for some $g\in Mor(\Delta )$, and $\gamma \in U\setminus V$.  Now since $\sigma \in <V\setminus U>$, $(Xg)\gamma \in <V\setminus U>$.  Now $Xg$ is a composition of face and degeneracy operators, and using the simplicial identities, we can always write $Xg=s_{l_1}...s_{l_k}\partial_{t_1}...\partial_{t_m}$.  Hence $\partial_{l_k}...\partial_{l_1}(Xg)\gamma =\partial_{t_1}...\partial_{t_m}\gamma \in <V\setminus U>$, and hence $\gamma \in \epsilon (<V\setminus U>)$.  But $\gamma \notin V$, and hence $\epsilon (<V\setminus U>)\not\subset V$.  The converse follows by symmetry.

\end{proof}

\begin{thm} (Contravariant Mayer-Vietoris sequence, version 1)
Let $X$ be a simplicial set with subsimplicial sets $U, V\subset X$  which cover $X$ and have good intersection.  Then there is a long exact sequence
\[ ...\leftarrow HA_c^n(U\cap V)\leftarrow HA_c^nU\oplus HA_c^nV\leftarrow HA_c^nX\leftarrow HA_c^{n-1}(U\cap V)\leftarrow ...\]

\end{thm}

\begin{proof}

Denote the obvious inclusions by $\iota^U:U\cap V\to U$, $\iota^V$$:U\cap V\to V$, $j^U:U\to X$, $j^V:V\to X$.  For any inclusion $\iota :Y\to Z$ of simplicial sets, there is an induced map $\iota_*:A_c^*Z\to A_c^*Y$ given by restriction: indeed, all we must show is that such restrictions vanish on all but finitely many non-degenerate simplices of $Y$, which follows from that fact that a simplex of $Y$ is degenerate in $Y$ if and only if it is degenerate in $Z$.  We claim we have a short exact sequence

\[ 0\to A_c^kX\stackrel{\theta_1}{\to}A_c^kU\oplus A_c^kV\stackrel{\theta_2}{\to}A_c^k(U\cap V)\to 0\]
for all $k$.  Define $\theta_1 $ by $\theta_1(\omega )=(\j^U_*\omega,-\j^V_*\omega )$.  Define $\theta_2(\omega_1,\omega_2)=\iota^U_*\omega_1+\iota^V_*\omega_2$.  Now $\theta_1$ is injective, because $U$ and $V$ cover $X$. To show $\theta_2$ is surjective, let $\omega \in A_c^k(U\cap V)$.  By the good intersection hypothesis, and \hyperref[bump_functions]{\ref*{bump_functions}}, there exists some $\phi \in A^kX$ such that $\phi |_{<U\setminus V>}=1$ and $\phi |_{<X\setminus U>}=0$.  Let $\phi_U=\phi $ and $\phi_V=1-\phi $.  Then these two functions form a partition of unity subordinate to the cover $\{ U,V\} $, and so $\theta_2(\phi_V\omega |_{U},\phi_U\omega |_{V})=\omega $.  Exactness at the middle term follows easily from the fact that we can glue forms which agree on their intersection. Hence the sequence is exact for all $k$, and the long exact sequence now follows as standard.

\end{proof}
There are in fact alternative conditions under which we can deduce a similar result.  For this, we recall the definition of a proper map of simplicial sets.
\begin{defn}
A map $f:X\to Y$ of simplicial sets is $\emph{proper}$ if for any finite subsimplicial set $Z\subset Y$, the subsimplicial set $f^{-1}(Z)\subset X$ is finite.

\end{defn}
\begin{rmk}
It is easily seen that all inclusions of simplicial sets are proper, and that any map $f:X\to Y$ is proper if and only $f^{-1}(<\sigma >)$ contains only finitely many non-degenerate simplices, for each non-degenerate simplex $\sigma \in Y$.

\end{rmk}

\begin{thm} (Contravariant Mayer-Vietoris sequence, version 2)

Suppose we have a pushout diagram

\begin{center}
\begin{tikzpicture}
\node (W) {$W$};
\node (U) [right of =W] {$U$};
\node (V) [below of =W] {$V$};
\node (X) [right of =V] {$X$};
\draw[->] (W) to node {$f$} (U);
\draw[->] (W) to node {$\iota $} (V);
\draw[->] (U) to node {$h$} (X);
\draw[->] (V) to node {$g$} (X);

\end{tikzpicture}\\
\end{center}
of simplicial sets, where $\iota $ is an inclusion, $f$ is proper and $V$ is locally finite.  Then $g$ and $h$ are proper maps, and there exists a long exact sequence
\[ ...\leftarrow HA_c^n(W)\leftarrow HA_c^nU\oplus HA_c^nV\leftarrow HA_c^nX\leftarrow HA_c^{n-1}(W)\leftarrow ...\]
which is natural in all the variables in the pushout.

\end{thm}

\begin{proof}
Since an inclusion of simplicial sets is a cofibration in the Kan-Quillen model structure and cofibrations are preserved under pushouts, $h$ is an inclusion (and hence is proper).  To show $g$ is proper, we use the structure of the pushout of simplicial sets, that is, $X\cong U\coprod_WV$ naturally.  For the rest of this proof we will identify $X$ and $U\coprod_W V$.  Now let $K\subset X$ be a finite subsimplicial set, and $\sigma \in K$ any simplex.  Then suppose $\sigma =f(\tau)$, some $\tau \in W$.  Then claim that $g^{-1}(\sigma )=\iota (f^{-1}(\sigma ))$.  Indeed, since $h$ is an inclusion, $f(\iota^{-1}(g^{-1}(\sigma )))=\{ \sigma \} $, hence $\iota (W)\cap g^{-1}(\sigma )\subset \iota (f^{-1}(\sigma ))$, but since $\sigma =f(\tau )$, $g^{-1}(\sigma )\subset \iota (W)$, hence $g^{-1}(\sigma )\subset \iota (f^{-1}(\sigma ))$.  For the reverse direction, observe that by commutativity of the pushout, $g\iota (f^{-1}(\sigma )={\sigma }$, and hence $\iota (f^{-1}(\sigma ))\subset g^{-1}(\sigma )$, and so we have the claimed equality.  This equality extends to show that $\iota (f^{-1}(K\cap U))=g^{-1}(K\cap U)$, and hence, $\iota (f^{-1}(K\cap U))\cup g^{-1}(K\setminus U)=g^{-1}(K)$.  But $g$ is injective on $g^{-1}(X\setminus U)$, and hence $g^{-1}(K\setminus U)$ has only finitely many non-degenerate simplices (non-degenerate in $V$), as does $\iota (f^{-1}(K\cap U))$, because $f$ is proper, and hence $g$ is proper.

\vspace{3mm}

We now prove the existence of the stated long exact sequence.  Since all the maps in the pushout diagram are proper, they all induce maps on $A_c(-)$ in the opposite direction.  By the proof of $14.1$ of \cite{bousgug}, there exists a short exact sequence of complexes
\[ 0\to AX\stackrel{(Ah,Ag)}{\to}AU\oplus AV\stackrel{A\iota -Af}{\to}AW\to 0\]
We also have the sequence
\[ 0\to A_cX\stackrel{(Ah,Ag)}{\to}A_cU\oplus A_cV\stackrel{A\iota -Af}{\to}A_cW\to 0\]
which we claim is exact also.  Notice that the maps in the second sequence are well-defined by properness.  $(Ah,Ag)$ is injective in the second sequence, since it is just a restriction of the map in the first sequence.  Now
\[ ker(A^n\iota -A^nf)=\{ (\Phi ,\Theta)\in A_c^n\oplus A_c^nV|A^nf(\Phi )=A^n\iota (\Theta )\} \]
so given any $(\Phi ,\Theta)\in ker(A^n\iota -A^nf)$, define $\Psi \in A^nX$ by
\begin{equation}
   \Psi |_{\sigma }=
    \begin{cases}
      \Phi |_{\sigma }, \text{if } \sigma \in U \\
      \Theta |_{\sigma }, \text{if } \sigma \in V
    \end{cases}
\nonumber
  \end{equation}
To show this is well defined, if $\sigma =f(\tau)\in U$, then
\[ \Theta |_{\tau }=A^nf(\Phi )|_{\tau}=\Phi |_{f(\tau )}=\Phi |_{\sigma }\]
as required.  So $(Ah,Ag)(\Psi )=(\Phi ,\Theta)$, and hence the sequence is exact at the middle term.  To show that $A^n\iota -A^nf$ is surjective in the second sequence, it suffices to show that $A^n\iota :A_c^nU\to A_c^nW$ is surjective (since $A_c^nU$ is naturally contained in $A_c^nU\oplus A_c^nV$).  Indeed, suppose $\omega \in A_{c}^nW$.  Then by the \hyperref[extension_prop]{extension property \ref*{extension_prop}} for $A$, there exists some extension $\omega_0\in A^nV$.  Now since $V$ is locally finite and $supp(\omega )$ is finite, $\epsilon(supp(\omega ))$ is also finite, and so by \hyperref[bump_functions]{\ref*{bump_functions}} there exists a bump function $\psi \in A^0(V)$ with $\psi |_{supp(\omega )}=1$ and $\psi |_{<V\setminus \epsilon (supp(\omega ))>}=0$, and hence $\psi \omega_0\in A_c^nV$ is an extension of $\omega $ as required.

\end{proof}
	\section{The PL compactly supported de Rham Theorem}

\begin{defn}
For a simplicial set $X$ with subsimplicial set $A$, we define the $\emph{relative polynomial de Rham complex}$ $A^*(X,A)$ by
\[ A^q(X,A)=\{ \Phi \in sSet(X,\nabla (*,q))| \Phi |_A=0\} \]

\end{defn}
In order to prove a de Rham theorem, we will need to use a model of singular cohomology which behaves well with respect to integration, which means we really don't want to have to think about degenerate simplices.  Thus we will use the $\emph{normalized}$ (or $\emph{Moore}$) complex of a simplicial Abelian group.  We will quickly say precisely what this is.

\begin{defn}
For a simplicial set $X$, the \emph{chain complex} $C_*X$ of $X$ is defined by
\[ C_nX=k[X_n]\]
with differential given by the alternating sum of the face maps
\[ \Sigma_{i=0}^n(-1)^i\partial_i:C_nX\to C_{n-1}X\]
$NC_*X$ will (temporarily) denote the $\emph{normalised chain complex}$ of $X$.  This is defined by
\[ NC_nX=k[X_n]/D(k[X_n])\]
where $D(Y_n)$ for a simplicial group $Y$ denotes the subgroup of $Y_n$ generated by the degenerate simplices.  The differential is induced by the differential on $C_*X$ (it is standard that this is well defined on the quotients by the degenerate simplices).\\
We define the corresponding (normalised) cochain complex as the dual over $k$ of the (normalised) chain complex, and for $A\subset X$, we define $C^*(X,A)$ to be the subobject of $C^*X$ of cochains which vanish on all simplices of $A$, and similarly for $NC^*(X,A)$.

\end{defn}

\begin{prop} (Eilenberg, Mac Lane.  Appears as Thm. $2.4$ in chap. III of \cite{goerssjardine})\\
For a simplicial set $X$ there is a natural inclusion $NC_*X\to C_*X$, which is a chain homotopy equivalence.  Hence by dualising, there is a natural chain homotopy equivalence $C^*X\to NC^*X$.

\end{prop}
This model allows Bousfield and Gugenheim to construct the de Rham natural transformation $A^*\to C^*$ using the integration we defined earlier.  We will repeat this now.

\vspace{3mm}

Write $\rho :A^*\to C^*$ for the natural transformation given by
\[ <\rho \omega ,\sigma >=\int \omega |_{\sigma }\]
where $\omega \in A^qX$ and $\sigma \in X_q$, for any simplicial set $X$.

\vspace{3mm}

Observe that if $\sigma $ is degenerate, then $\sigma $$=$$s_j\sigma '$, for some $\sigma '$, and $\omega |_{\sigma}$$=$$s_j\omega |_{\sigma '}$$=$$0$, because $\omega |_{\sigma '}\in \nabla (q-1,q)=0$.  Hence $\rho_X(\omega )$ vanishes on degenerate simplices, and so $\rho $ in fact maps into $NC^*X$.  It is easy to check that $\rho $ is a well defined natural transmormation.\\
We now have the PL de Rham theorem

\begin{thm} (2.2 and 3.4 of \cite{bousgug})\\
$\rho $ induces a multiplicative homology isomorphism
\[ \rho _*HA^*X\to HNC^*X\simeq HC^*X\]
for any simplicial set $X$.

\end{thm}
An easy corollary is

\begin{cor}
For any pair of simplicial sets $(X,A)$ with $A\subset X$, the chain map $\rho_{(X,A)}:A^*(X,A)\to NC^*(X,A)$, defined by restricting $\rho_X$, induces a multiplicative homology isomorphism
\[ \rho _*:HA^*(X,A)\to HNC^*(X,A)\cong HC^*(X,A)\]

\end{cor}
\begin{proof}
We have the following diagram of chain complexes, with exact rows

\begin{tikzpicture}
\node at (-3,0) (01) {$0$};
\node at (0,0) (AXA) {$A^*(X,A)$};
\node at (3,0) (AX) {$A^*X$};
\node at (6,0)(AA) {$A^*A$};
\node at (9,0) (02) {$0$};
\node at (-3,-3) (03) {$0$};
\node at (0,-3) (CXA) {$NC^*(X,A)$};
\node at (3,-3) (CX) {$NC^*X$};
\node at (6,-3) (CA) {$NC^*A$};
\node at (9,-3) (04) {$0$};
\draw[->] (01) to node {} (AXA);
\draw[->] (AXA) to node {} (AX);
\draw[->] (AX) to node {} (AA);
\draw[->] (AA) to node {} (02);
\draw[->] (03) to node {} (CXA);
\draw[->] (CXA) to node {} (CX);
\draw[->] (CX) to node {} (CA);
\draw[->] (CA) to node {} (04);
\draw[->] (AXA) to node {$\rho_{(X,A)}$} (CXA);
\draw[->] (AX) to node {$\rho_X$} (CX);
\draw[->] (AA) to node {$\rho_A$} (CA);

\end{tikzpicture}\\
and so the vertical maps given by $\rho $ induce a map of long exact homology sequences, and $\rho_X$ and $\rho_A$ induce homology isomorphisms by the de Rham theorem (2.2 of \cite{bousgug}), and so by the five lemma, $\rho_{(X,A)}$ also induces a homology isomorphism.

\vspace{3mm}

To see that the isomorphism is multiplicative, observe that it factors as the composition
\[ HA^*(X,A)\to HA^*(X/A,*)\to HC^*(X/A,*)\to HC^*(X,A)\]
where the first and last maps are the isomorphisms by the induced canonical multiplicative maps on the level of chains.  The fact that the middle map is multiplicative follows from the non-compactly supported de-Rham theorem, and the fact that the reduced homologies of $X/A$ are canonically isomorphic to the respective homologies, except in degree $0$ where they are both $0$.

\end{proof}
From now on we will identify $NC^*$ and $C^*$, and their compactly supported versions also.

\vspace{3mm}

We now show that $A_c^*$ has an alternative construction as a direct limit of relative cohomology groups (analagous to the well known construction of $C_cX$ in this way).

\vspace{3mm}

Let $X$ be a simplicial set.  Observe that the collection of finite simplicial subsets of $X$ forms a directed system, since it is closed under finite unions.  Moreover, if $L\subset K\subset X$ are finite simplicial subsets, we get an evident induced map $HA^*(X,<X\setminus L>)\to HA^*(X,<X\setminus K>)$, and so the collection $\{ HA^*(X,<X\setminus K>)| K\subset X\text{ is finite}\} $ is naturally a directed system also, and so we can form the direct limit
\[ colim_K(HA^*(X,<X\setminus K>))\]
(or indeed we can form the direct limit of the $A^*(X,<X\setminus K>)$)
as $K$ ranges over all finite subsimplicial sets of $X$.  We note as well that we consider this as a colimit in the category of graded Abelian groups, or graded rings, and the underlying complex will be the same.  We now have the following results

\begin{lma}
\label{relative}
The natural map
\[ \eta :HA_c^*X\to colim_K(HA^*(X,<X\setminus K>))\]
is an isomorphism of graded rings.

\end{lma}
\begin{proof}

$\eta $ is defined as follows.  Given $\omega \in A_c^qX$, there exists some finite $K\subset X$ such that $\omega |_K=0$, and so $\omega $ naturally belongs to some
\[ A^*(X,<X\setminus K>)\]
and is also a cocycle in this complex.  Hence $\omega $ naturally represents some element of the direct limit in the lemma, which is what we define $\eta (\omega )$ to be.  To show $\eta $ is well defined on the level of homology, suppose $\omega =\omega '+d\gamma $, for some $\gamma \in A_c^{q-1}X$.  Then $\omega $ and $\omega '$ are both cocycles in some common $A^*(X,<X\setminus K>)$, and $K$ can also be chosen so that $\gamma $ belongs to the complex, and it is clear that $\omega =\omega '+d\gamma $ in $A^*(X,<X\setminus K>)$ also, and hence $\omega $ and $\omega '$ are cohomologous in $colim_L(A^*(X,<X\setminus L>))$, and so since homology commutes with direct limits, they represent the same element of the direct limit in the lemma, hence $\eta $ is well defined.\\
Showing that $\eta $ is injective and surjective just amounts to noticing that any cochain in $\omega \in A_c^*X$ is closed if and only if its image in
\[ colim_K(A^*(X,<X\setminus K>))\]
is closed (and again using that homology commutes with direct limits).  The fact that it is multiplicative follows easily from how the multiplication on the colimit is defined.

\end{proof}
It is standard that the same two lemmas hold with $A^*$ replaced by $C^*$.  Hence we now have

\begin{thm}(PL compactly supported de Rham Theorem)\\
The restriction $\rho_c:A_c^*X\to C_c^*X$ induces a multiplicative isomorphism on cohomology.

\end{thm}

\begin{proof}
By \hyperref[relative]{Lemma \ref*{relative}} and the variant for $C^*$, we have isomorphisms
\[ HA_c^*X\to colim_K(HA^*(X,<X\setminus K>))\]
\[ HC_c^*X\to colim_K(HC^*(X,<X\setminus K>))\]
and so by the relative de Rham theorem (and naturality of the above isomorphisms), $\rho_c$ is a homology isomorphism.

\end{proof}

\bibliographystyle{plain}

\bibliography{thesisreferences}

\begin{thebibliography}{1}

\bibitem{infhom}
H-J. Baues and A.~Quintero.
\newblock {\em {Infinite Homotopy Theory}}.
\newblock Springer, 2001.

\bibitem{bousgug}
A.K. Bousfield and V.K.A.M. Gugenheim.
\newblock {On PL de Rham Theory and Rational Homotopy Type}.
\newblock {\em Memoirs of the American Mathematical Society}, 8(179), 1976.

\bibitem{goerssjardine}
P.G. Goerss and J.F. Jardine.
\newblock {\em {Simplicial Homotopy Theory}}.
\newblock Modern Birkh{\"{a}}user Classics, 2009.

\bibitem{haraguchi}
T.~Haraguchi.
\newblock {Long Exact Sequences for de Rham cohomology of Diffeological
  Spaces}.
\newblock {\em arXiv:1404.1127}, 2014.

\bibitem{quillenrht}
D.G. Quillen.
\newblock {Rational Homotopy Theory}.
\newblock {\em The Annals of Mathematics}, 90(2):205--295, 1969.

\bibitem{sullivaninfinitesimal}
D.~Sullivan.
\newblock {Infinitesimal Computations in Topology}.
\newblock {\em Publications math\'{e}matiques de l'I.H.{\'{E}}.S},
  47(1):269--331, 1977.

\end{thebibliography}

\end{document}